\documentclass[12pt, letterpaper]{article}
\usepackage{dirtytalk}
\usepackage{amsmath}
\usepackage{amsfonts}
\usepackage{amssymb}
\usepackage{amsthm}
\usepackage{breqn}
\usepackage{setspace}
\usepackage{fullpage}
\usepackage{enumitem}
\usepackage{comment}
\usepackage{hyperref}
\usepackage{bbm}
\usepackage{tikz}
\usepackage{enumitem}
\usepackage{mathtools} 
\usepackage[utf8]{inputenc}
\usepackage[english]{babel}
\usepackage{mathtools}
\usetikzlibrary{patterns,arrows,decorations.pathreplacing}
\usepackage{xcolor}

\newtheorem{observation}{Observation}
\newtheorem{lemma}{Lemma}
\newtheorem{theorem}{Theorem} 
 
\newtheorem{definition}{Definition}
\newtheorem{claim}{Claim}

\newtheorem{conjecture}{Conjecture}

\newcommand{\p}{\mathcal{P}}

\newcommand{\hh}{\mathcal{H}}

\DeclareMathOperator{\ex}{ex}

\usepackage{authblk}

\usetikzlibrary{patterns}
\title{On the Halin Tur\'an number of short cycles}
\author{Addisu Paulos\\ addisu.wmeskel@aau.edu.et}
\affil{Addis Ababa University  \\ Addis Ababa, Ethiopia}
\date{}
\begin{document}
\maketitle
\begin{abstract}
A Halin graph is a graph constructed by embedding a tree with no vertex of degree two in the plane and then adding a cycle to join the tree's leaves. The Halin Tur\'an number of a graph $F$, denoted as $\ex_{\hh}(n,F)$, is the maximum number of edges in an $n$-vertex Halin graph.  In this paper, we give the exact value of $\ex_{\mathcal{H}}(n,C_4)$, where $C_4$ is a cycle of length 4. We also pose a conjecture for the Halin Tur\'an number of longer cycles.  
\end{abstract}
\subsubsection*{Keywords: Halin graphs,Tur\'an number, Halin Tur\'an number}

\section{Introduction}
Let $F$ be a fixed graph. A graph $G$ is called $F$-\textit{free} if it contains no isomorphic copy of $F$ as a subgraph. For the graph $F$ and a positive integer $n$, the \emph{Tur\'an number} of $F$, denoted by $\ex(n,F)$, is the maximum number of edges in an $n$-vertex $F$-free graph, i.e.,  
\begin{align*}
\ex(n,F)=\max \{e(G): \mbox{$G$ is an $n$-vertex $F$-free graph} \}.
\end{align*}
One of the classical results in extremal graph theory is the Tur\'an's Theorem~\cite{1}, which gives the exact value $\ex(n,K_r)$, where $K_r$ is an $r$-vertex complete graph. This result is the generalization of the Mantel's Theorem~\cite{2} for the case of $K_3$.  A major breakthrough in the study of Tur\'an number of graphs came in 1966, with the proof of the famous theorem by Erd\H{o}s, Stone and Simonovits~\cite{3,4}. They determined an asymptotic value of the Tur\'an number of any fixed non-bipartite graph $F$. In particular, they proved  $\ex(n,F)=\left(1-\frac{1}{\chi(F)-1}\right){n\choose 2}+o(n^2)$, where $\chi(F)$ is the chromatic number of $F$.
Since these results, researchers have been interested in working on the Tur\'an number of class of bipartite (degenerate) graphs and extremal problems in a particular family of graphs. In 2016, Dowden~\cite{5} initiated the study of Tur\'an-type problems in the family of planar graphs.

\begin{definition}
Let $F$ be a fixed graph and $n$ be a positive integer. The \textbf{planar Tur\'an number of $F$}, denoted by $\ex_{\p}(n,F)$, is the maximum number of edges an $n$-vertex $F$-free planar graph contains, i.e., 
\begin{align*}
\ex_{\p}(n,F)=\max \{e(G): \mbox{$G$ is an $n$-vertex $F$-free planar graph} \}.
\end{align*}
\end{definition}
Dowden~\cite{5} determined sharp upper bounds of $\ex_{\p}(n,C_4)$ and $\ex_{\mathcal{P}}(n,C_5)$, where $C_k$ is a cycle of length $k$.  
\begin{theorem}~\cite{5}
\begin{enumerate}
\item For $n\geq 4$, $$\ex_{\mathcal{P}}(n,C_4)\leq \frac{15(n-2)}{7}.$$ 
\item For $n\geq 11$, $$\ex_{\mathcal{P}}(n,C_5)\leq \frac{12n-33}{5}.$$
\end{enumerate}
\end{theorem}
Extending the Dowden's result,  Lan, Shi and Song~\cite{6} obtained an upper bound for $\text{ex}_{\p}(n,C_6)$, and later Ghosh, Gy\H{o}ri, Martin, Paulos and Xiao~\cite{7}, improved the bound and gave a sharp upper bound with some interesting constructions realizing their bound.  However $\ex_{\p}(n,C_k)$ is still open for general $k$. We refer \cite{9,10, 11, 8} for a quick survey and conjectures on planar Tur\'an numbers of graphs.
\begin{theorem}\cite{7} \label{sw}
For all $n\geq 18,$
$$\ex_{\p}(n,C_6)\leq \frac{5}{2}n-7.$$ 
\end{theorem}
 Recently, Fang and Zhai~\cite{12} initiated the study of Tur\'an numbers in the family of outerplanar numbers. For a positive integer $n$ and fixed graph $F$, the outerplanar Tur\'an number of $F$, denoted by $\text{ex}_{\mathcal{OP}}(n,F)$, is the maximum number of edges in an $n$-vertex outerplanar graph containing no isomorphic copy of $F$ as a subgraph. They completely determined the outerplanar Tur\'an numbers of cycles and paths.

In this paper, we initiate the study of Tur\'an number of cycles in the family of Halin graphs. A \textit{Halin graph} $H$ is constructed as follows: Start with a tree $T$ in which each non-leaf has degree at least 3, i.e., every non-leaf of $T$ is with degree at least $3$. Embed the tree
in the plane in a planar fashion and then add new edges to form a cycle $C$ containing all the leaves of T in such a way that the resulting graph $H$ is planar. We write $H=T\cup C$, and we call $T$ and $C$ respectively as \emph{characteristic tree} and \emph{outer cycle} of the Halin graph $H$.  

Halin graphs were studied by Halin~\cite{13}. A Halin graph has at least four vertices. The wheel graph, $W_n$, is an
example of Halin graph with the characteristic tree being a star on $n$ leaves.  Halin graphs are, edge-minimal and $3$-connected~\cite{14}. Every edge of a Halin graph is part of some Hamiltonian cycle~\cite{15}. 

\begin{definition}
Let $n$ be a positive integer and $F$ be a fixed graph. The Halin Tur\'an number of $F$, denoted by $\text{ex}_{\hh}(n,F)$, is the maximum number of edges in an $n$-vertex $F$-free Halin graph, i.e., 
\begin{align*}
\ex_{\hh}(n,F)=\max \{e(H): \mbox{$H$ is an $n$-vertex $F$-free Halin graph} \}.
\end{align*}
\end{definition}
Bondy and Lovasz~\cite{16} have shown that Halin graphs are almost pancyclic.  More precisely, they showed that if a Halin graph $H$ on $n$ vertices does not have any vertex of degree three in its characteristic tree, then it has all cycles of length $\ell$, where, $3\leq \ell\leq n$.  If the characteristic tree contains a vertex of degree three, then cycles of all lengths will still be there with a possible exception of an even-length cycle. In a different study, He and Liu explored the maximum count of short paths in a Halin graph, as discussed in~\cite{44}.

The almost pancyclic property of Halin graphs makes them interesting from a theoretical perspective, as it implies that these graphs are highly connected and can be used to model a wide variety of complex systems and phenomena. As a result, much research in this area focuses on developing efficient algorithms and techniques for analyzing the structure and properties of Halin graphs.

Concerning cycles, it is still interesting to study and distinguish the extremal graph structures and the Halin Tur\'an number of cycles of even length. In this paper, we determine the exact value of the Halin Tur\'an number of the $4$-cycle, and later we pose our conjecture for longer cycles. The following theorem states our main result.   
\begin{theorem}\label{t1}
For $n\geq 16$, 
\begin{equation*}\ex_{\hh}(n,C_4)=
\begin{cases}
\frac{5}{3}(n-1), & 3|(n-1),\\
\frac{5}{3}(n-2)+1, & 3|(n-2),\\
\frac{5}{3}(n-3)+3, & 3|(n-3).
\end{cases}
\end{equation*}
\end{theorem}

The following notations and terminologies are needed. Let $G$ be a graph. We denote the vertex and the edge sets of $G$ by $V(G)$ and $E(G)$ respectively. The number of vertices and edges in $G$ respectively are denoted by $v(G)$ and $e(G)$.  For a vertex $v$ in $G$, the degree of $v$ is denoted by $d_{G}(v)$. We may omit the subscript if the underlying graph is clear.  The set of all vertices in $G$ which are adjacent to $v$ is denoted as $N_G(v)$ or simply $N(v)$ when the underlying graph is clear. For the sake of simplicity, we use the terms $k$-cycle and $k$-path to mean a cycle of length $k$ and a path of length $k$ respectively. We denote a $k$-cycle with vertices $v_1,\ v_2,\dots ,\ v_k$ in sequential order by $(v_1, \ v_2,\ \dots, \ v_k,\ v_1)$.  We denote a $k$-path with vertices $v_0,\ v_2,\dots ,\ v_k$ in sequential order by $(v_0,\ v_1,\dots, v_k)$. A $(u,\ v)$-path is a path with end vertices $u$ and $v$. Given a $k$-path $(v_0,\ v_1,\ \dots,\ v_k)$, we may describe $v_1$ and $v_{k-1}$ as \textit{semi-pendant} vertices of the path. For a plane graph $G$, the length of a cycle $C$ in $G$ is denoted by $|C|$. Similarly, the size of a face $F$ in $G$ is denoted by $|F|$.  

Let $H$ be a Halin graph and $T$ be its characteristic tree. A non-leaf $v\in V(T)$ is an \textit{interior vertex} if every vertex in $N_T(v)$ is not a leaf. A non-leaf $u\in  V(T)$ is a \textit{branching vertex} if it has at most one non-leaf in $N_T(u)$. A \textit{semi-branching vertex} $w\in V(T)$ is a non-leaf that is neither an interior nor a branching vertex. Sometimes we may call a leaf in $T$ a \emph{pendant vertex}.

\section{Proof of Theorem~\ref{t1}}
The following lemmas and observations are important to complete the proof of the theorem.
\begin{lemma}\label{l1}
Let $H$ be a $C_4$-free Halin graph and $T$ be its characteristic tree. For a longest path $L$ in $T$, each semi-pendant vertex of $L$ is a branching vertex and is adjacent to only two leaves.\end{lemma} 
\begin{proof}
Let $L=(v_0,\ v_1,\ v_2,\ \dots,\ v_k)$. Since $L$ is a longest path, $v_{k-1}$ can not be adjacent to a non-leaf vertex except $v_{k-2}$.  Moreover, from the definition of a Halin graph, $d_T(v_{k-1})\geq 3$. Thus $N_T(v_{k-1})\backslash\{v_{k-2}\}$ contains leaves. If $N_T(v_{k-1})\backslash\{v_{k-2}\}$ contains three leaves, say $u_1,\ u_2$, and $u_3$ in sequential order in counterclockwise direction, then $H$ contains a $4$-cycle, namely $(v_{k-1},\ u_1,\ u_2,\ u_3,\ v_{k-1})$, and hence a contradiction.   
\end{proof}
\begin{lemma}\label{kjh}
Let $H=T\cup C$ be a Halin graph, and $u_1$ and $u_2$ be leaves in $T$ such that $u_1u_2\in E(C)$. Let $F$ be the bounded face incident to $u_1u_2$. If $\mathcal{C}$ is a cycle in $H$ containing $u_1u_2$ we have, $|\mathcal{C}|\geq |F|$.   
\end{lemma}
\begin{proof}
Let the boundary cycle of $F$ be $(u_1,\ u_2, \ u_3,\ \dots, u_k,\ u_1)$. Denote $R=\{u_3,\ u_4,\ \dots, \ u_k\}$.  Each vertex in $R$ is not a leaf in $T$, since $H$ is a Halin graph and the degree of each vertex is at least three. For each vertex $u\in R$, there is a unique leaf $u'$ in $T$ such that we have a $(u,\ u')$-path with the set of interior vertices disjoint from $R$. We may call $u'\ $s as \textit{child-pendant} vertices of $u$. Any $(u_1,\ u_2)$-path other than the edge $u_1u_2$ must contain either $u$ or some child-pendant vertex $u'$ for each $u\in R$. This implies $|\mathcal{C}|\geq |F|$. 
\end{proof}
\begin{lemma}\label{l21}
Let $H$ be a Halin graph with a characteristic tree $T$. Let $e=uv\in E(T)$ such that both $u$ and $v$ are non-leaf in $T$. If $F_1$ and $F_2$ are the two bounded faces incident to $e$, then for a cycle $\mathcal{C}$ in $T$ containing $e$, then  $|\mathcal{C}|\geq \min\{|F_1|,\ |F_2|\}$.  
\end{lemma}
\begin{proof}
Since $u$ and $v$ are non-leaf and $H$ is a Halin graph, then $d_T(u),\ d_T(v)\geq 3$. Therefore, we have vertices $u_1,\ u_2\in N(u)$ and $v_1,\ v_2\in N(v)$ such that $(u_1,\ u,\ v, \ v_1)$ and $(u_2,\ u,\ v,\ v_2)$ are paths incident to $F_1$ and $F_2$ respectively. Moreover, we have vertices $u_1',\ v_1'$ and $u_2',\ v_2'$, which are leaves in $T$ such that $e_1=u_1'v_1'$ and $e_2=u_2'v_2'$ are edges incident to $F_1$ and $F_2$ respectively. Notice that, $u_1'$ can be $u_1$ and $v_1'$ can be $v_1$, and  similarly for $u_2'$ and  $v_2'$ with $u_2$ and  $v_2$. Clearly, $\mathcal{C}$ contains either $e_1$ or $e_2$, but not both. If $\mathcal{C}$ contains $e_1$, then by Lemma~\ref{kjh}, $|\mathcal{C}|\geq |F_1|$. Moreover if $\mathcal{C}$ contains $e_2$, the $|\mathcal{C}|\geq |F_2|$. Therefore, $|C|\geq \min\{|F_1|,\ |F_2|\}$. 
\end{proof}
\begin{lemma}\label{mainlemma1}
Let $H$ be an $n$-vertex $C_4$-free Halin graph with the characteristic tree $T$. If $T$ contains a semi-branching vertex of degree at least 4, then there is an $(n-1)$-vertex $C_4$-free Halin graph $H'$ such that $e(H)=e(H')+2$. 
\end{lemma}
\begin{proof}
Let $C$ be the outer cycle of $H$. Let $v\in V(T)$, with $d_H(v)\geq 4$, be a semi-branching vertex and $u\in N_T(v)$ be a leaf. Let the path $(u_1,\ u,\ u_2)$ be the portion of $C$ in the clockwise direction and denote $F_1$ and $F_2$ as faces in $H$ incident to the paths $(v,\ u,\ u_1)$ and $(v,\ u,\ u_2)$ respectively. It can be seen that either $|F_1|$ and $|F_2|$ is at least 5. Indeed, since $H$ is a $C_4$-free graph, no face is of size $4$. On the other hand if $|F_1|=|F_2|=3$, then $(v,\ u_1,\ u,\ u_2,\ v)$ is in $H$ and this contradicts the $C_4$-free assumption of $H$.  Now obtain the graph $H'$ by deleting $u$ and joining the vertices $u_1$ and $u_2$ with an edge. $H'$ is a Halin graph with characteristic tree $T'=T-u$, as $d_{T'}(v)\geq 3$ and $d_{T'}(w)=d_{T}(w)$ for every $w\in V(T)\backslash \{v\}$. Let $C'$ be the characteristic tree and the outer cycle of $H'$. $H'$ is $C_4$-free as the bounded face, say $F$, incident to the edge $u_1u_2$ is of size at least 5, and by Lemma~\ref{kjh} the boundary cycle of $F$ is the smallest cycle containing $u_1u_2$.
\end{proof}
\begin{lemma}\label{rfr}
Let $H$ be an $n$-vertex $C_4$-free Halin graph with characteristic tree $T$. Let $(u,\ v,\ w)$ be a path in $T$ such that $v$ is a semi-branching vertex with $d_T(v)=3$. If the bounded face incident to the path is with size at least $6$, then there is an $(n-2)$-vertex $C_4$-free Halin graph $H'$ such that $e(H)=e(H')+3$. 
\end{lemma}
\begin{proof}
Let $v'\in N(v)$ and $F_1,\ F_2$ and $F_3$ as the faces incident to the paths $(u,\ v,\ w),\ (u,\ v,\ v')$ and $(v',\ v,\ w)$ respectively. By assumption $|F_1|\geq 6$. Since $H$ is $C_4$-free and $v$ is a semi-branching vertex, then $|F_2|, \ |F_3|\geq 5$. Denote $u'$ and $w'$ as the leaves in $T$ such that $v'u'$ is incident to $F_2$ and $v'w'$ is incident to the face $F_3$. Let $H'$ be a graph obtained from $H$ by deleting $v$ and adding the edges $uw$ and $u'w'$. It can be checked that $H'$ is an $(n-2)$-vertex Halin graphs, with the two faces incident to $uw$ with size at least $5$ and at least $6$, and hence by Lemma~\ref{kjh} $H'$ contains no $4$-cycle and $e(H)=e(H')+3$. 
\end{proof}
\begin{lemma}\label{l22}
Let $H$ be an $n$-vertex  $C_4$-free Halin graph with characteristic tree $T$. Let $e\in E(T)$ such that its end vertices are non-leaf in $T$. If the two faces incident to $e$ are with size at least $6$, then there is an $(n-1)$-vertex $C_4$-free Halin graph, $H'$, such that $e(H)=e(H')+1.$ 
\end{lemma}
\begin{proof}
Denote $H=T\cup C$, where $C$ is the outer cycle of $H$. Let $e=vu$ and $F_1$ and $F_2$ be the two bounded faces in $H$ incident to $e$. $u$ and $v$ by assumption are non-leaf, and hence $d_T(u)\geq 3$ and $d_T(v)\geq 3$. Let $T'$  be the graph obtained after contracting $e$ in $T$. Clearly, $T'$ an $(n-1)$-vertex tree and a leaf in $T'$ is a leaf in $T$. Moreover, for each non-leaf vertex $w\in V(T'),\ d_{T'}(w)\geq 3$. Therefore by contracting $e$ in $H$ we get a Halin graph $H'=T'\cup C$. 

Since $|F_1|,\ |F_2|\geq 6$, then by Lemma~\ref{l21}, for each cycle $\mathcal{C}$ containing $e$ we have $|\mathcal{C}|\geq 6$. Hence, by contracting $e$, every cycle in $H'$ is with no $4$-cycle. Therefor, $H'$ is an $(n-1)$-vertex $C_4$-free Halin graph. This completes the proof of Lemma~\ref{l22}. 
\end{proof}
\begin{lemma}
For $n\geq 16$, we have  
\begin{equation*}\ex_{\hh}(n,C_4)\geq 
\begin{cases}
\frac{5}{3}(n-1), & 3|(n-1),\\
\frac{5}{3}(n-2)+1, & 3|(n-2),\\
\frac{5}{3}(n-3)+3, & 3|(n-3).
\end{cases}
\end{equation*}
\end{lemma}
\begin{proof} We give extremal constructions to verify the bounds. First, we give constructions of the characteristic tree of the Halin graph, when $n=16,\ 17$, and $n=18$. For the sake of simplicity, we may call the trees as \textit{base-tree} and denote them by $T_{16},\ T_{17}$ and $T_{18}$. Denote also the corresponding Halin graphs by $H_{16},\ H_{17}$ and $H_{18}$ respectively. It is easy to see the Halin graphs are $C_4$-free.    
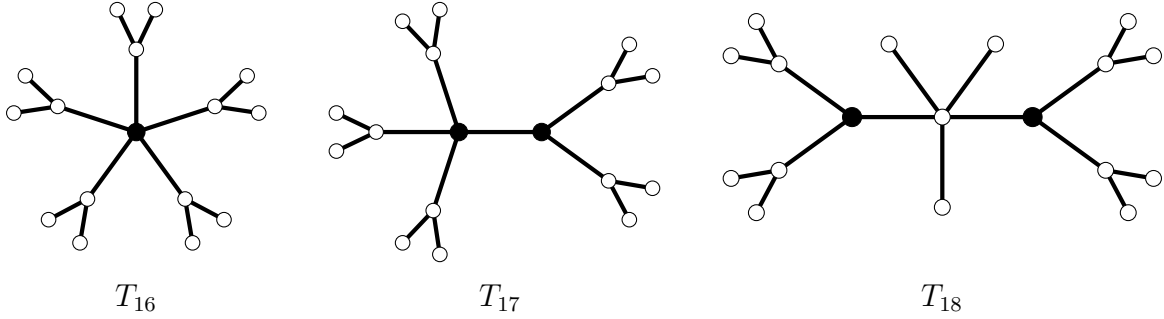
\begin{figure}[h]
\centering
\begin{tikzpicture}[scale=0.11]
\draw[ultra thick](0,0)--(0,10)(0,0)--(9.5,3.1)(0,0)--(-9.5,3.1)(0,0)--(5.9,-8.1)(0,0)--(-5.9,-8.1)(9.5,3.1)--(14.8,2.3)(9.5,3.1)--(13.4,6.8) (-9.5,3.1)--(-14.8,2.3)(-9.5,3.1)--(-13.4,6.8)(0,10)--(2.3,14.8)(0,10)--(-2.3,14.8)(5.9,-8.1)--(10.6,-10.6)(5.9,-8.1)--(6.8,-13.4)(-5.9,-8.1)--(-10.6,-10.6)(-5.9,-8.1)--(-6.8,-13.4);
\draw[fill=black](0,0)circle(30pt);
\draw[fill=white](0,10)circle(25pt);
\draw[fill=white](9.5,3.1)circle(25pt);
\draw[fill=white](-9.5,3.1)circle(25pt);
\draw[fill=white](5.9,-8.1)circle(25pt);
\draw[fill=white](-5.9,-8.1)circle(25pt);
\draw[fill=white](14.8,2.3)circle(25pt);
\draw[fill=white](13.4,6.8)circle(25pt);
\draw[fill=white](-14.8,2.3)circle(25pt);
\draw[fill=white](-13.4,6.8)circle(25pt);
\draw[fill=white](2.3,14.8)circle(25pt);
\draw[fill=white](-2.3,14.8)circle(25pt);
\draw[fill=white](10.6,-10.6)circle(25pt);
\draw[fill=white](6.8,-13.4)circle(25pt);
\draw[fill=white](-10.6,-10.6)circle(25pt);
\draw[fill=white](-6.8,-13.4)circle(25pt);
\node at (0,-20) {$T_{16}$};
\end{tikzpicture}\quad\quad
\begin{tikzpicture}[scale=0.11]
\draw[ultra thick](-5,0)--(5,0)--(13.1,5.9)--(18.4,6.8)(13.1,5.9)--(15.6,10.6) (5,0)--(13.1,-5.9)--(18.4,-6.8)(13.1,-5.9)--(15.6,-10.6)(-5,0)--(-8.1,9.5)--(-11.8,13.4)(-8.1,9.5)--(-7.3,14.8)(-5,0)--(-8.1,-9.5)--(-11.8,-13.4)(-8.1,-9.5)--(-7.3,-14.8)(-5,0)--(-15,0)--(-19.8,2.3)(-15,0)--(-19.8,-2.3);
\draw[fill=black](-5,0)circle(30pt);
\draw[fill=black](5,0)circle(30pt);
\draw[fill=white](13.1,5.9)circle(25pt);
\draw[fill=white](13.1,-5.9)circle(25pt);
\draw[fill=white](-8.1,9.5)circle(25pt);
\draw[fill=white](-8.1,-9.5)circle(25pt);
\draw[fill=white](-15,0)circle(25pt);
\draw[fill=white](18.4,6.8)circle(25pt);
\draw[fill=white](15.6,10.6)circle(25pt);
\draw[fill=white](18.4,-6.8)circle(25pt);
\draw[fill=white](15.6,-10.6)circle(25pt);
\draw[fill=white](-11.8,13.4)circle(25pt);
\draw[fill=white](-7.3,14.8)circle(25pt);
\draw[fill=white](-11.8,-13.4)circle(25pt);
\draw[fill=white](-7.3,-14.8)circle(25pt);
\draw[fill=white](-19.8,2.3)circle(25pt);
\draw[fill=white](-19.8,-2.3)circle(25pt);
\node at (0,-20) {$T_{17}$};
\end{tikzpicture}\quad\quad
\begin{tikzpicture}[scale=0.12]
\draw[ultra thick](-5,0)--(5,0)--(13.1,5.9)--(18.4,6.8)(13.1,5.9)--(15.6,10.6) (5,0)--(13.1,-5.9)--(18.4,-6.8)(13.1,-5.9)--(15.6,-10.6) (-23.1,5.9)--(-28.4,6.8)(-23.1,5.9)--(-25.6,10.6) (-23.1,-5.9)--(-25.6,-10.6)(-23.1,-5.9)--(-28.4,-6.8)(-5,0)--(-15,0)--(-23.1,5.9)(-15,0)--(-23.1,-5.9)(-5,0)--(0.9,8.1)(-5,0)--(-10.9,8.1)(-5,0)--(-5,-10);
\draw[fill=white](-5,-10)circle(25pt);
\draw[fill=black](-15,0)circle(30pt);
\draw[fill=white](-5,0)circle(25pt);
\draw[fill=black](5,0)circle(30pt);
\draw[fill=white](13.1,5.9)circle(25pt);
\draw[fill=white](13.1,-5.9)circle(25pt);
\draw[fill=white](18.4,6.8)circle(25pt);
\draw[fill=white](15.6,10.6)circle(25pt);
\draw[fill=white](18.4,-6.8)circle(25pt);
\draw[fill=white](15.6,-10.6)circle(25pt);
\draw[fill=white](-23.1,5.9)circle(25pt);
\draw[fill=white](-23.1,-5.9)circle(25pt);
\draw[fill=white](-28.4,6.8)circle(25pt);
\draw[fill=white](-25.6,10.6)circle(25pt);
\draw[fill=white](-28.4,-6.8)circle(25pt);
\draw[fill=white](-25.6,-10.6)circle(25pt);
\draw[fill=white](0.9,8.1)circle(25pt);
\draw[fill=white](-10.9,8.1)circle(25pt);
\node at (-5,-20) {$T_{18}$};
\end{tikzpicture}
\caption{Characteristic trees of Halin graphs on 16, 17, and 18 vertices}
\label{f1}
\end{figure}

Now let $n\geq 19$. We define an $n$-vertex Halin graphs $H_{16}^n,\ H_{17}^n$ and $H_{18}^n$ based on the base-trees $T_{16}, \ T_{17}$ and $T_{18}$ as follows. The star $K_{1,3}$, which is shown in Figure~\ref{f90}, is an important component in describing the constructions. For simplicity reasons, we call it \textit{star}. Notice the dark-spotted vertices in both the base-trees and the star.   

For  $n\equiv 0~(\mod \ 3)$, the Halin graph $H_{18}^n$ is obtained by having $\frac{n-18}{3}$ copies of the star and identifying any of the dark-spotted vertices of $T_{18}$ and the dark-spotted vertex of the star.  Similarly, when $n\equiv 1~(\mod \ 3)$ and $n\equiv 2~(\mod \ 3)$, we respectively get $H_{16}^n$ and $H_{17}^n$ by having $\frac{n-16}{3}$ and $\frac{n-17}{3}$  copies of the star and identifying the dark-spotted vertices of the corresponding base-trees and the star.

It is easy to see that the Halin graphs, $H_{16}^n,\ H_{17}^n$ and $H_{18}^n$ are $C_4$-free. Moreover it is easy to calculate and check that $e(H_{16}^n)=\frac{5}{3}(n-1), \ e(H_{17}^n)=\frac{5}{3}(n-2)+1$ and $e(H_{18}^n)=\frac{5}{3}(n-3)+3$. Therefore, for $n\equiv 0~(\mod \ 3)$,\ $n\equiv 1~(\mod \ 3)$  and $n\equiv 2~(\mod \ 3)$, we have $\ex_\hh(n,C_4)\geq e(H_{18}^n), \ \ex_\hh(n,C_4)\geq e(H_{16}^n) $, and $\ex_\hh(n,C_4)\geq e(H_{17}^n)$ respectively. 
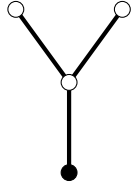
\begin{figure}[h]
\centering
\begin{tikzpicture}[scale=0.12]
\draw[ultra thick](-5,0)--(-10.9,8.1)(0.9,8.1)--(-5,0)--(-5,-10);
\draw[fill=black](-5,-10)circle(25pt);
\draw[fill=white](-5,0)circle(25pt);
\draw[fill=white](0.9,8.1)circle(25pt);
\draw[fill=white](-10.9,8.1)circle(25pt);
\end{tikzpicture}
\caption{Star $K_{1,3}$}
\label{f90}
\end{figure}
\end{proof}
\begin{observation}\label{ob1}
\emph{ Let $H$ be an $n$-vertex $C_4$-fee Halin graph with a characteristic tree $T$. Let $L=(v_0,\ v_1,\ v_2,\ \dots,\ v_{k-2},\ v_{k-1},\ v_k)$ be a longest path in $T$.} \emph{From Lemma~\ref{l1}, both $v_1$ and $v_{k-1}$ are branching vertices, and each of them is adjacent to two leaves. Denote the leaf, other than $v_0$, adjacent to $v_1$ by $v_0'$. Denote also the leaf, other than $v_k$, adjacent to $v_{k-1}$ by $v_k'$. Since each non-leaf in $T$ is with a degree at least $3$, there must be a vertex, say $u$, adjacent to $v_2$ such that either both $v_0v_1$ and $v_2u$ or both $v_1v_0'$ and $v_2u$ are incidents to the same bounded face in $H$. Without loss of generality assume the latter case holds. It can be seen that $u$ can not be a leaf in $T$. Otherwise, $(v_1,\ v_2,\ u,\ v_0',\ v_1)$ is a $4$-cycle in $H$, which is a contradiction. Hence $u$ is non-leaf in the characteristic tree. Therefore, $d_T(u)\geq 3$. From the assumption that $L$ is of maximum length in $T$, $u$ is adjacent to exactly two leaves, and say $u_1$ and $u_2$. For a similar argument, $v_{k-2}$ is adjacent to a non-leaf $w$, which is adjacent to two leaves $w_1$ and $w_2$, see Figure~\ref{f2}.}
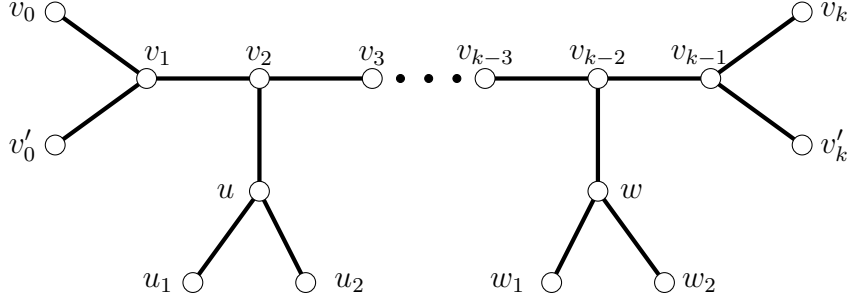
\begin{figure}[h]
\centering
\begin{tikzpicture}[scale=0.15]
\draw[ultra thick](5,0)--(15,0)--(25,0)--(33.1,5.9)(25,0)--(33.1,-5.9)(15,-10)--(20.9,-18.1)(15,-10)--(10.9,-18.1) (-5,0)--(-15,0)--(-25,0)--(-33.1,5.9)(-25,0)--(-33.1,-5.9)(-15,-10)--(-20.9,-18.1)(-15,-10)--(-10.9,-18.1)(-15,0)--(-15,-10)(15,0)--(15,-10);
\draw[fill=white](-15,0)circle(25pt);
\draw[fill=white](15,0)circle(25pt);
\draw[fill=white](-5,0)circle(25pt);
\draw[fill=white](5,0)circle(25pt);
\draw[fill=white](-25,0)circle(25pt);
\draw[fill=white](-33.1,5.9)circle(25pt);
\draw[fill=white](-33.1,-5.9)circle(25pt);
\draw[fill=black](-2.5,0)circle(10pt);
\draw[fill=black](2.5,0)circle(10pt);
\draw[fill=black](0,0)circle(10pt);
\draw[fill=white](25,0)circle(25pt);
\draw[fill=white](33.1,5.9)circle(25pt);
\draw[fill=white](33.1,-5.9)circle(25pt);
\draw[fill=white](20.9,-18.1)circle(25pt);
\draw[fill=white](10.9,-18.1)circle(25pt);
\draw[fill=white](-20.9,-18.1)circle(25pt);
\draw[fill=white](-10.9,-18.1)circle(25pt);
\draw[fill=white](15,-10)circle(25pt);
\draw[fill=white](-15,-10)circle(25pt);
\node at (36,5.9) {$v_k$};
\node at (36,-5.9) {$v_k'$};
\node at (24,2) {$v_{k-1}$};
\node at (15,2) {$v_{k-2}$};
\node at (5,2) {$v_{k-3}$};
\node at (18,-10) {$w$};
\node at (24,-18.1) {$w_2$};
\node at (7,-18.1) {$w_1$};
\node at (-36,5.9) {$v_0$};
\node at (-36,-5.9) {$v_0'$};
\node at (-24,2) {$v_{1}$};
\node at (-15,2) {$v_{2}$};
\node at (-5,2) {$v_{3}$};
\node at (-18,-10) {$u$};
\node at (-24,-18.1) {$u_1$};
\node at (-7,-18.1) {$u_2$};
\end{tikzpicture}
\caption{Distribution of vertices on a longest path of a Halin graph}
\label{f2}
\end{figure}
\emph{When $k\geq5$, $S=\{v_0,v_0',v_1,v_2,u,u_1,u_2,v_k,v_k',v_{k-1},v_{k-2},w,w_1,w_2\}$ gives the $14$ labeled vertices in $T$. If $k=4$, $v_2$ and $v_{k-2}$ are identical vertices and the stars attached at $v_2$ and $v_{k-2}$ could be identical.}
\end{observation}
\begin{lemma}\label{fbv}
Let $H$ be an $n$-vertex $C_4$-free Halin graph, where $n\geq 19$. Then 
\begin{equation*}\ex_{\hh}(n,C_4)\leq 
\begin{cases}
\frac{5}{3}(n-1), & 3|(n-1),\\
\frac{5}{3}(n-2)+1, & 3|(n-2),\\
\frac{5}{3}(n-3)+3, & 3|(n-3).
\end{cases}
\end{equation*}
\end{lemma}
\begin{proof}
 Our proof relies on induction on the number of vertices. The base cases are shown in the upcoming section. Let $L=(v_0,\ v_1,\ v_2,\ \dots,\ v_k)$ be a longest path in $T$. It is easy to check that $k\geq 4$. Next, we prove the following sequence of lemmas as part of the proof. 
\begin{claim}\label{clm1}
If $k=4$, then $3|(n-1)$ and $e(H)=\frac{5}{3}(n-1).$
\end{claim}
\begin{proof}
From observation~\ref{ob1}, $S=\{v_0,v_0',v_1,v_2,u,u_1,u_2,v_3,v_4,v_4'\}$. For each vertex $v\in V(T)\backslash S$ and is incident to $L$, $v\in N(v_2)$. Moreover, $v$ is not a leaf in $T$. Indeed, suppose for contradiction $v$ is a leaf. Since $L$ is the longest path, then the two faces incident to the edge $v_2v$ are with size either $3$ or $4$. The latter, can not happen as $H$ is $C_4$-free. Thus we may assume both faces are with size three. Hence we get two triangles sharing the same edge $v_2v$. However, this also results in a $4$-cycle, which is a contradiction. Hence, each vertex in $T$ adjacent to $v_2$ is a non-leaf. Since $L$ is a longest path, each vertex adjacent to $v_2$ is a branching vertex. That means the vertex is adjacent to two pendant vertices. Therefore, $H$ is obtained by identifying the dark-spotted vertex of $\frac{n-7}{3}$ copies of stars with $v_2$. It can be checked that $e(H)=\frac{5}{3}(n-1)$, and this completes the proof of Claim~\ref{clm1}.
\end{proof}
\begin{claim}\label{cdc1}
If $k=5$, then $3|(n-2)$ and $e(H)=\frac{5}{3}(n-2)+1.$
\end{claim}
\begin{proof}
From Observation~\ref{ob1}, $S=\{v_0,v_0',v_1,v_2,u,u_1,u_2,v_3,w,w_1,w_2,v_4,v_5,v_5'\}$. We verify that each vertex $v\in V(T)\backslash S$ incident to $v_2$ or $v_3$ is a branching vertex. Indeed, without loss of generality assume $v\in N(v_2)$. 
%and $v$ is on the upper side of $L$. Suppose for contradiction $v$ is a leaf. 
Since $L$ is a longest path in $T$, the faces incident to $v_2v$ and located on its left side must be either a $3$-face or a $4$-face. The latter can not happen, as $H$ is a $C_4$-free graph. Thus we may assume the face is a $3$-face and let the leaf forming the $3$-face be $v'$, i.e., the $3$-face is $(v_2,\ v,\ v',\ v_2)$. For the same reason, we have a leaf $v''$ adjacent to $v_2$ such that $(v_2,\ v',\ v'',\ v_2)$ is the $3$-face incident to the edge $v'v$. However this results a $4$-cycle $(v_2,\ v,\ v',\ v'', v_2)$, which is a contradiction. Therefore, each vertex in $V(T)\backslash S$ adjacent to $v_2$ or $v_3$ is a branching vertex. This implies, $H$ is obtained by identifying the black-spotted vertex of $\frac{n-8}{3}$ copies of the star to either $v_2$ or $v_3$. It can be checked that $e(H)=\frac{5}{3}(n-2)+1.$  This completes the proof of Claim~\ref{cdc1}.    
\end{proof}
\begin{claim}\label{vfg}
For $k\geq 6$, 
%and $3|(n-3)$ or $3|(n-1)$, 
then $H$ meets either the conditions of Lemma~\ref{mainlemma1} or the conditions of Lemma~\ref{rfr} or the conditions of Lemma~\ref{l22}.    
\end{claim}

\begin{proof}
Consider the longest path $L=(v_0,\ v_1, v_2,\ v_4,\dots,\ v_k)$. As $L$ is a longest path,  $v_1$ is a branching vertex and hence it is adjacent to two leaves where $v_0$ is one of the two vertices. Let the other vertex be $v_0'$. From the degree condition of Halin graph, $d_T(v_2)\geq 3$. Moreover, every vertex adjacent to $v_2$ is not a leaf. Since again $L$ is a longest path, each vertex adjacent to $v_2$ must be a branching vertex. If $v_3$ is a semi-branching vertex of degree at least 4, then $H$ satisfies the condition of Lemma~\ref{mainlemma1} and we are done. So we may assume that $v_3$ is not a semi-branching vertex or a semi-branching vertex with $d_T(v_3)=3$. In the former case, the edge $v_2v_3$ is an edge with the property that its end vertices are non-leaf and the two faces incident to the edge are with size at least $6$, and hence $H$ satisfies the conditions of Lemma~\ref{l22}. In the latter case, since $L$ is a longest path in $T$, $v_2$ is not a semi-branching vertex. Hence, the path $(v_2,\ v_3,\ v_4)$ is with a size of at least 6, and hence $H$ meets conditions of Lemma~\ref{rfr}. This completes the proof of Claim~\ref{vfg}.
\end{proof}
Notice that we finish the proof of Lemma~\ref{fbv} if $3|(n-3)$ or $3|(n-1)$. Indeed, if conditions of Lemma~\ref{mainlemma1} or  Lemma~\ref{l22} happen, then $e(H)\leq e(H')+2$, where $H'$ is an $(n-1)$-vertex $C_4$-free Halin graph. If $3|(n-1)$, then by induction we have $e(H)=e(H')+2\leq \left(\frac{5}{3}\left[(n-1)-3\right]+3\right)+2=\frac{5}{3}(n-1)$ and we are done. On the other hand if $3|(n-3)$, then by induction we have, $e(H)=e(H')+2\leq \left(\frac{5}{3}\left[(n-1)-2\right]+1\right)+2=\frac{5}{3}(n-3)+3$ and we are done. On the other hand, suppose conditions of Lemma~\ref{rfr} meet by $H$. In this case, $e(H)=e(H')+3$, where $H'$ is an $(n-2)$-vertex $C_4$-free Halin graph. If $3|(n-1)$, then $e(H)=e(H')+3\leq \left(\frac{5}{3}\left[(n-2)-2\right]+1\right)+3\leq \frac{5}{3}(n-1)$. If $3|(n-3)$, then $e(H)=e(H')+3\leq \frac{5}{3}\left[(n-2)-1\right]+3=\frac{5}{3}(n-3)+3$, and we are again done by induction.

Next, we give our argument on how we finish the proof when $k\geq 6$ and $3|(n-2)$. Since $H$ is a Halin graph and $v_{k-3}$ is a non-leaf vertex, $d_T(v_{k-3})\geq 3$. 

If $d_T(v_{k-3})=3$, then there is a vertex, say $x$ such that $x\in N_T(v_{k-3})$. If $x$ is a leaf in $T$, then it can be seen that the path $(v_{k-4},\ v_{k-3}, \ v_{k-2})$ is incident to a face of size at six. Then by Lemma~\ref{rfr} we have an $(n-2)$-vertex $C_4$-free Halin graph $H'$ such that $e(H)=e(H')+3$. Thus, by induction, $e(H)=e(H')+3\leq \left(\frac{5}{3}\left[(n-2)-3\right]+3\right)+3=\frac{5}{3}(n-2)+1$, and we are done by induction. On the other hand, if $x$ is not a leaf in $T$, then again it can be seen that the two bounded faces incident to the edge $v_{k-3}v_{k-2}$ are with size at least 6 and hence by Lemma~\ref{l22} we have an $(n-1)$-vertex $C_4$-free Halin graph $H'$ such that $e(H)=e(H')+1$. This implies by induction $e(H)=e(H')+1\leq \frac{5}{3}\left[(n-1)-1\right]+1=\frac{5}{3}(n-2)+1$ and we are done by induction.   

Now we may assume that $d_T(v_{k-3})\geq 4$. From Observation~\ref{ob1}, we have a branching vertex $w\in N_T(v_{k-2})$ such that the path $(v_{k-1},\ v_{k-2},\ w)$ is incident to a bounded face in $H$. Since $L$ is a longest path in $T$, every vertex $N_T(v_{k-2})\backslash \{v_{k-3}\}$ is a branching vertex. Let $F_1$ and $F_2$ be the two bounded faces incident to the edge $v_{k-3}v_{k-2}$. Notice that we have a  unique pair of leaves incident to each bounded face in $H$. Denote $v^u_1$ and $w^u_1$ be the leaves such that the edge $v^u_1w^u_1$ is incident to $F_1$ and $w^u_1$ is a leaf  adjacent to the vertex in $N(v_{k-2})\backslash \{v_{k-3}\}$. Similarly denote $v^l_1$ and $w^l_1$ be leaves such that the edge $v^l_1w^l_1$ is incident to $F_2$ and $w^l_1$ is a leaf  adjacent to the vertex in $N(v_{k-2})\backslash \{v_{k-3}\}$. Notice that $w^u_1$ and $w^l_1$ could be $v_k$ or $w_1$ as discussed in Observation~\ref{ob1}. Moreover, notice that both $|F_1|$ and $|F_2|$ are at least 5. If both $|F_1|$ and $|F_2|$ are with size at least $6$, then we finish the proof by induction using Lemma~\ref{l22} considering the edge $v_{k-3}v_{k-2}$.  
So we may assume one of the two faces is with size $5$. Without loss of generality assume $|F_1|=5$, and hence $u^u_1\in N(v_{k-3})$. 
%Let $P$ be a $(v_{k-2}',\ v_{k-2}'')$-path which is the portion of the outer cycle of $H$ in clockwise direction.
Let $F_3$ be the bounded face in $H$ incident to the path $(v_1,\ v_2, \ u)$ as discussed in Observation~\ref{ob1}. We perform the following three operations on $H$ step by step to get a new and equivalent Halin graph $H'$ to $H$, i.e., $e(H')=e(H)$. 
\begin{enumerate}
\item Delete the edges $\ v_{k-3}v_{k-2}$,\ $v^u_1w^u_1$ and $v^l_1w^l_1$ from $H$. The resulting disconnected graph has two components and let $C^1$ and $C^2$ be the components containing $v_{k-3}$ and $v_{k-2}$ respectively.
\item Place the component $C^2$ in $F_3$ keeping its shape. Apply rigid motions on $C^2$ so that the pair of vertices $\{v_2,\ v_{k-2}\}$, $\{v_0',\ w^u_1\}$ and $\{w^l_1, u_1\}$ are joined by an edge after deleting the edge $v_0'u_1$ in $C^1$.  
\item Join the pair of vertices $\{v^u_1,\ v^l_1\}$ with an edge and denote the resulting graph by $H'$.
\end{enumerate}
Since $d_T(v_{k-3})\geq 4$ we have, $d_{H'}(v_{k-3})\geq 3$. Moreover $v(H')=v(H)$ and all the leaves of $T$ form the outer face of $H'$. Thus, $H'$ is a Halin graph equivalent to $H$. However, it may happen that $H'$ may contain a $C_4$. If a $4$-cycle exists in $H'$, then it must contain an edges in $\{v_0'w^u_1,\ w^l_1u_1,\ v_2v_{k-2},\ v^u_1v^l_1\}$. Since the two faces incident to the edge $v_2v_{k-2}$ in $H'$ are of size at least 6, then using Lemma~\ref{l21} for any cycle $\mathcal{C}$ containing an edge in $\{v_0'w^u_1,\ w^l_1u_1,\ v_2v_{k-2}\}$ we have, $|\mathcal{C}|\geq 6$. This implies, if the Halin graph $H'$ contains a $4$-cycle, then it must contain the edge $v^u_1v^l_1$. 

If $H'$ is $C_4$-free graph, then we can finish the proof by induction using Lemma~\ref{l22}. Indeed, the two bounded faces incident to the edge $v_2v_{k-2}$ in $H'$ are of size at least 6. From Lemma~\ref{l22} we have $e(H')=e(H'')+1$, where $H''$ is an $(n-1)$-vertex $C_4$-free Halin graph associated to $H'$ in the lemma. Therefore, $e(H)=e(H')+1\leq \frac{5}{3}\left[(n-1)-1\right]+1=\frac{5}{3}(n-2)+1$.

Now we assume $H'$ contains a $4$-cycle. As explained earlier, the cycle contains $v^u_1v^l_1$. Such a $4$-cycle happens when $v^l_1\in N(v)$, where $v$ is in $N_H(v_{k-3})$, or $v^l_1\in N_H(v_{k-4})$ or $v^l_1\in N_H(v_{k-3})$ and at least one of the edges in $\{v_{k-3}v^u_1,\ v_{k-3}v^l_1\}$ is incident to a $3$-face in $H$. Let the face, other than the $F_1$, and incident to $v_{k-3}v^u_1$ be $F_4$. Denote the associated leaf by $v^u_2$ such that $v^u_2v^u_1$ is incident to $F_4$. We distinguish the three situations separately to complete the proof.
\subsubsection*{Case 1: when $v^l_1\in N_H(v)$, where $v\in N_H(v_{k-3})$ and $v$ is a branching vertex }
%If $F_2$ is incident to an edge $v_{k-3}v$, where $v$ is a vertex not on $L$, then $v^l_1\in N(v)$. In this case $(v_{k-3},\ v^u_1,\ v^l_1,\ v)$ is the only $4$-cycle in $H'$. In addition, 
%The face other than $F_1$ and incident to $v_{k-3}v^u_1$ can be a $3$-face. If the face is a $3$-face, then denote the associated pendant in $N(v_{k-3})$ by $v^u_2$. In other words, $(v_{k-3},\ v^u_2,\ v^u_1,\ v_{k-3})$ is the $3$-face.
Notice that $F_4$ may or may not be a $3$-face. We finish the proof by induction.  Indeed, if $F_4$ is not a $3$-face, then from Lemma~\ref{mainlemma1} using the semi-branching vertex $v_{k-3}$ and then applying Lemma~\ref{l22} using the edge $v_{k-3}v_{k-2}$, we get an $(n-2)$-vertex $C_4$-free Halin graph, $H^*$ such that $e(H)=e(H^*)+3\leq \left(\frac{5}{3}\left[(n-2)-3\right]+3\right)+3=\frac{5}{3}(n-2)+1$. On the other hand, if $F_4$ is a $3$-face, then applying Lemma~\ref{mainlemma1} on the the semi-branching vertex $v_{k-3}$ twice and then using Lemma~\ref{l22} on the edge $v_{k-3}v_{k-2}$ we get an $(n-3)$-vertex $C_4$-free Halin graph $H^*$. Moreover we have $e(H)=e(H^*)+5\leq \left(\frac{5}{3}\left[(n-3)-2\right]+1\right)+5=\frac{5}{3}(n-2)+1$, and we are done by induction.  

\subsubsection*{Case 2: when $v^l_1\in N_H(v_{k-4})$}
Actually, this may happen when $k\geq 7$. If $F_4$ is not a $3$-face, then we can still finish the proof by induction using Lemma~\ref{mainlemma1} considering $v_{k-3}$ as a semi-branching vertex and then applying Lemma~\ref{l22} using the edge $v_{k-3}v_{k-2}$ as the two faces incident to the edge are with size at least 6. Observe that the resulting graph is an $(n-2)$-vertex $C_4$-free Halin graph which only miss 3 edges. The same argument holds to finish the proof by induction if $F_4$ is a $3$-face and $d_T(v_{k-3})\geq 5.$ In this case the resulting graph is an $(n-3)$-vertex $C_4$-free Halin graph which only miss 5 edges. On the other hand, if $d_{T}(v_{k-3})=4$ and $F_4$ is a $3$-face, then we apply Lemma~\ref{mainlemma1} on $H$ using the $v^u_2$ and then Lemma~\ref{rfr} using the leaf $v^u_1$, so that we get an $(n-3)$-vertex $C_4$-free Halin graph which only miss 5 edges. This we can finish the proof by induction as shown in Case~1 above.

\subsubsection*{Case 3: when $v^l_1\in N_H(v_{k-3})$ and at least one of the edges in $\{v_{k-3}v^u_1,\ v_{k-3}v^l_1\}$ is incident to a $3$-face in $H$.}
Let the face other than the $F_2$ and incident to $v_{k-3}v^l_1$ be $F_5$. Let $v^l_2$ be a leaf in $H$ such that the edge $v^l_1v^l_2$ is incident to $F_5$. $F_5$ may or may not be a $3$-face. If both $F_4$ and $F_5$ are $3$-faces, then $d_{H'}(v_{k-3})\geq 5$. In this case, delete the vertices $v^u_1$ and $v^l_1$ from $H'$ and then add the edge $v^u_2v^l_2$. This leaves an $(n-2)$-vertex $C_4$-free Halin graph, say $H^*$, with 4 edges reduced. Next apply Lemma~\ref{l22} on the edge $v_2v_{k-2}$ on $H^*$, the resulting graph becomes an $(n-3)$-vertex $C_4$-free Halin graph missing only 5 edges from the original graph $H$. With this, we can complete the proof by induction as given in Case~1 above. Finally, assume $F_4$ is a $3$-face but not $F_5$. In this case, $d_{H'}(v_{k-3})\geq 4$. It can be checked that deleting $v^u_1$ and adding the edge $v^u_2v^l_1$ in $H'$ leaves an $(n-1)$-vertex $C_4$-free Halin graph which misses only two edges. Applying Lemma~\ref{l22} on the edge $v_2v_{k-2}$ results an $(n-2)$-vertex $C_4$-free Halin graph which loses only $3$ edges from the original graph $H$. Again in this case we can finish the induction as stated in Case~1. This completes the proof of Lemma~\ref{fbv}. 
\end{proof} 
\section{Basis of the induction steps}
To finish the proof by induction, we verify the bound when $n=16,\ 17$, and $18$. The following lemmas give the details. %Later in the appendix, we give the library of the Halin graphs with the indicated number of vertices.

\begin{lemma}\label{fvc}
For a $16$-vertex $C_4$-free Halin graph $H$, $e(H)\leq 25$, i.e., $e(H)\leq \frac{5}{3}(n-1)$ where $n=16.$
\end{lemma}
\begin{proof}
We prove the statement addressing different situations for which the length of the longest path a characteristic tree may possibly contain. Let $T$ be the characteristic tree of $H$. Let $L$ be a longest path in  $T$ and $k$ be its length. It is very trivial to check that there is no $H$ when $k\leq 3.$
\begin{claim}\label{cl0}
 $k$ is at most 6. 
\end{claim}
\begin{proof}
Let $k=7$. From Observation~\ref{ob1} we have, $L=(v_0,\ v_1,\ v_2,\ \dots,\ v_7)$ and $S=\{v_0,v_0',v_1,v_2,u, u_1, u_2,v_7,v_7',v_{6},v_{5},w,w_1,w_2\}$. However, such vertex assignment leaves the non-leaf vertices $v_3$ and $v_4$ in the characteristic tree with degree 2, which is a contradiction to the definition of a Halin graph. Hence the maximum possible choice of $k$ is $6$. This completes the proof of Claim~\ref{cl0}.    
\end{proof}
\begin{claim}\label{cl2}
If $k=6$, then $e(H)=24.$     
\end{claim}
\begin{proof}
Denote $L=(v_0,\ v_1,\ v_2,\ v_3,\ v_4,\ v_5,\ v_6)$. Based on the notations Observation~\ref{ob1} we have, $S=\{v_0,v_0',v_1,v_2,u,u_1,u_2,v_6,v_6',v_5,v_4,w,w_1,w_2\}$. Since $v_3$ is a non-leaf in $T$, it must be adjacent with a vertex, say $v_3'$. Since $|S\cup \{v_3,v_3'\}|=16$, then every vertex of $T$ is now labeled. Clearly $P=\{v_0,v_0',u_1,u_2,v_3',w_1,w_2,v_6,v_6'\}$ is the set of all pendant vertices of $T$. Therefore, $e(H)=e(T)+|P|=15+9=24.$ This completes the proof of Claim~\ref{cl2}. It is easy to see that Figure~\ref{f3d} is the only characteristic tree $T$ meeting the case.
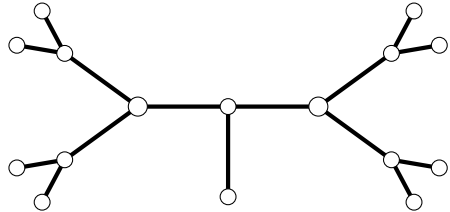
\begin{figure}[h]
\centering
\begin{tikzpicture}[scale=0.12]
\draw[ultra thick](-5,0)--(5,0)--(13.1,5.9)--(18.4,6.8)(13.1,5.9)--(15.6,10.6) (5,0)--(13.1,-5.9)--(18.4,-6.8)(13.1,-5.9)--(15.6,-10.6) (-23.1,5.9)--(-28.4,6.8)(-23.1,5.9)--(-25.6,10.6) (-23.1,-5.9)--(-25.6,-10.6)(-23.1,-5.9)--(-28.4,-6.8)(-5,0)--(-15,0)--(-23.1,5.9)(-15,0)--(-23.1,-5.9)(-5,0)--(-5,-10);
\draw[fill=white](-5,-10)circle(25pt);
\draw[fill=white](-15,0)circle(30pt);
\draw[fill=white](-5,0)circle(25pt);
\draw[fill=white](5,0)circle(30pt);
\draw[fill=white](13.1,5.9)circle(25pt);
\draw[fill=white](13.1,-5.9)circle(25pt);
\draw[fill=white](18.4,6.8)circle(25pt);
\draw[fill=white](15.6,10.6)circle(25pt);
\draw[fill=white](18.4,-6.8)circle(25pt);
\draw[fill=white](15.6,-10.6)circle(25pt);
\draw[fill=white](-23.1,5.9)circle(25pt);
\draw[fill=white](-23.1,-5.9)circle(25pt);
\draw[fill=white](-28.4,6.8)circle(25pt);
\draw[fill=white](-25.6,10.6)circle(25pt);
\draw[fill=white](-28.4,-6.8)circle(25pt);
\draw[fill=white](-25.6,-10.6)circle(25pt);
\end{tikzpicture}
\caption{The characteristic tree of a Halin graph on 16 vertices}
\label{f3d}
\end{figure}   
\end{proof}
\begin{claim}\label{cl3}
There is no $H$ when $k=5$.
\end{claim}
\begin{proof}
Denote $L=(v_0,\ v_1,\ v_2,\ v_3,\ v_4,v_5)$. From the discussion in Observation~\ref{ob1}, $S=\{v_0,v_0',v_1,v_2,u,u_1,u_2,v_5,v_5',v_4,v_3,w,w_1,w_2\}$. There are two remaining vertices which are not assigned yet. Let this vertex be labeled as $x_1$ and $x_2$. For a clear reason, the vertices are adjacent to either $v_2$ or $v_3$. Without loss of generality suppose $v_2$ is such a vertex. In this case, there is a vertex in $\{x_1,x_2\}$, say $x_1$, is adjacent to $v_2$ such that $xv_2$ and $uu_2$ are incident to the same $4$-face in $H$. But this is a contradiction to the fact that $H$ is a $C_4$-free Halin graph. This completes the proof of Claim~\ref{cl3}.
\end{proof}
\begin{claim}\label{cl4}
If $k=4$, then $e(H)=25.$
\end{claim}
\begin{proof}
Let $L=(v_0,\ v_1,\ v_2,\ v_3.\ v_4)$. By Observation~\ref{ob1}, $S=\{v_0,v_0', v_1,v_2,u,u_1,u_2,v_4, v_4',v_3\}$. There are $6$ vertices remaining, and label the vertices as $x_1,x_2,\dots,x_6$. Let $x_1$ be adjacent to $v_2$, and suppose for contradiction $x_1$ be a leaf in $T$. It can be checked that the two faces incident to the edge $v_2x_1$ in $H$ are of size $4$ or $3$. Moreover, non of the faces are of size 4. On the other hand, if both faces are with size 3, then again a $4$-cycle will be obtained as the two $3$-cycles sharing an edge forms a $4$-cycle. This is again a contradiction. Thus, each vertex in $R=\{x_1,x_2,\dots, x_6\}$ adjacent to $v_2$ is not pendant, and a vertex in $R$ adjacent to $v_2$ is again adjacent to two pendant vertices in $R$. Therefore, $T$ is obtained by identifying the dark-spotted vertex of three stars shown in Figure~\ref{f90} with the vertex $v_2$. The resulting graph is $T_{16}$ which is shown in Figure~\ref{f1}. It can be calculated that $e(H)=e(T)+10=25$. This completes the proof of Claim~\ref{cl4} and Lemma~\ref{fvc}.    
\end{proof}  
\end{proof}
\begin{lemma}\label{vbv}
For a $17$-vertex $C_4$-free Halin graph $H$, $e(H)\leq 26$, i.e., $e(H)\leq \frac{5}{3}(n-2)+1$ where $n=17.$
\end{lemma}
\begin{proof}
We give similar proof to the one given in Lemma~\ref{l1}. Let $H$ be a $C_4$-free Halin graph on $17$ vertices, and $T$ denote its corresponding characteristic tree. Let $L$ be a longest path in $T$ with length $k$. It can be checked that $k\geq 4$.
\begin{claim}\label{cl5}
$k$ is at most $6$. 
\end{claim}
\begin{proof}
By Observation~\ref{ob1} we have, $S=\{v_0,v_0',v_1,v_2,u,u_1,u_2,v_k,v_k',v_{k-1},v_{k-2}, w,w_1,w_2\}$. Since $T$ has 17 vertices and $|S|=14$, three vertices are still not used. It can be seen that if $k\geq 7$, there exists a non-leaf vertex in $L$ with degree $2$, and this violates the definition of Halin graphs. Therefore $k\leq 6$. This completes the proof of Claim~\ref{cl5}.
\end{proof}
\begin{claim}\label{cl6}
If $k=6$, then $e(H)=26$.
\end{claim}
\begin{proof}
In this case $S=\{v_0,v_0',v_1,v_2,u,u_1,u_2,v_6,v_6',v_5,v_4,w,w_1,w_2\}$. Notice that the vertex $v_3$ in $L$ is a non-leaf in $T$, and there are two vertices in $T$, say $x_1$ and $x_2$, that are not in $L$. It is easy to check that non of the vertices is incident to $v_2$ or $v_4$, and both vertices are pendant and incident to $v_3$. Clearly $H$ is $C_4$-free and $e(H)=e(T)+10=26=\frac{5}{3}(n-2)+1$ where $n=17$. There are two possible non-isomorphic characteristic trees, see Figure~\ref{f3}. This completes the proof of Claim~\ref{cl6}. 
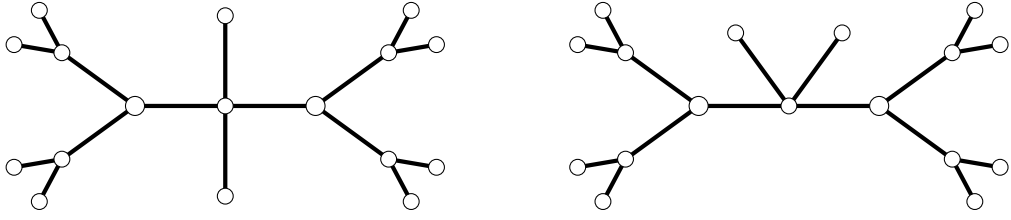
\begin{figure}[h]
\centering
\begin{tikzpicture}[scale=0.12]
\draw[ultra thick](-5,0)--(5,0)--(13.1,5.9)--(18.4,6.8)(13.1,5.9)--(15.6,10.6) (5,0)--(13.1,-5.9)--(18.4,-6.8)(13.1,-5.9)--(15.6,-10.6) (-23.1,5.9)--(-28.4,6.8)(-23.1,5.9)--(-25.6,10.6) (-23.1,-5.9)--(-25.6,-10.6)(-23.1,-5.9)--(-28.4,-6.8)(-5,0)--(-15,0)--(-23.1,5.9)(-15,0)--(-23.1,-5.9)(-5,10)--(-5,0)--(-5,-10);
\draw[fill=white](-5,10)circle(25pt);
\draw[fill=white](-5,-10)circle(25pt);
\draw[fill=white](-15,0)circle(30pt);
\draw[fill=white](-5,0)circle(25pt);
\draw[fill=white](5,0)circle(30pt);
\draw[fill=white](13.1,5.9)circle(25pt);
\draw[fill=white](13.1,-5.9)circle(25pt);
\draw[fill=white](18.4,6.8)circle(25pt);
\draw[fill=white](15.6,10.6)circle(25pt);
\draw[fill=white](18.4,-6.8)circle(25pt);
\draw[fill=white](15.6,-10.6)circle(25pt);
\draw[fill=white](-23.1,5.9)circle(25pt);
\draw[fill=white](-23.1,-5.9)circle(25pt);
\draw[fill=white](-28.4,6.8)circle(25pt);
\draw[fill=white](-25.6,10.6)circle(25pt);
\draw[fill=white](-28.4,-6.8)circle(25pt);
\draw[fill=white](-25.6,-10.6)circle(25pt);
\end{tikzpicture}\qquad\qquad
\begin{tikzpicture}[scale=0.12]
\draw[ultra thick](-5,0)--(5,0)--(13.1,5.9)--(18.4,6.8)(13.1,5.9)--(15.6,10.6) (5,0)--(13.1,-5.9)--(18.4,-6.8)(13.1,-5.9)--(15.6,-10.6) (-23.1,5.9)--(-28.4,6.8)(-23.1,5.9)--(-25.6,10.6) (-23.1,-5.9)--(-25.6,-10.6)(-23.1,-5.9)--(-28.4,-6.8)(-5,0)--(-15,0)--(-23.1,5.9)(-15,0)--(-23.1,-5.9)(-5,0)--(0.9,8.1)(-5,0)--(-10.9,8.1);
\draw[fill=white](-15,0)circle(30pt);
\draw[fill=white](-5,0)circle(25pt);
\draw[fill=white](5,0)circle(30pt);
\draw[fill=white](13.1,5.9)circle(25pt);
\draw[fill=white](13.1,-5.9)circle(25pt);
\draw[fill=white](18.4,6.8)circle(25pt);
\draw[fill=white](15.6,10.6)circle(25pt);
\draw[fill=white](18.4,-6.8)circle(25pt);
\draw[fill=white](15.6,-10.6)circle(25pt);
\draw[fill=white](-23.1,5.9)circle(25pt);
\draw[fill=white](-23.1,-5.9)circle(25pt);
\draw[fill=white](-28.4,6.8)circle(25pt);
\draw[fill=white](-25.6,10.6)circle(25pt);
\draw[fill=white](-28.4,-6.8)circle(25pt);
\draw[fill=white](-25.6,-10.6)circle(25pt);
\draw[fill=white](0.9,8.1)circle(25pt);
\draw[fill=white](-10.9,8.1)circle(25pt);
\end{tikzpicture}
\caption{Characteristic tree of a $17$-vertex Halin graphs}
\label{f3}
\end{figure}
\end{proof}
\begin{claim}\label{cl6}
If $k=5$, then $e(H)=26$.
\end{claim}
\begin{proof}
By Observation~\ref{ob1}, $S=\{v_0,v_0',v_1,v_2,u,u_1,u_2,v_5,v_5',v_4,v_3,w,w_1,w_2\}$. There are three vertices in $T$ which are not labeled yet. Denote the vertices as $x_1,x_2$ and $x_3$, If one of the three vertices is adjacent to $v_2$ (similarly $v_3$) and is a pendant vertex in $T$, then all the remaining two vertices are pendant vertices and adjacent to $v_2$ or $v_3$. One of the three edges forms a $4$-face containing the edge in $\{v_1v_0,v_2u,v_3w,v_4v_5\}$. But this results in a contradiction, as $H$ is $C_4$-free. Therefore, the only possible situation that $T$ exists is when the three vertices are connected to the $L$ by identifying the black-spotted vertex of the star, see Figure~\ref{f90}, with either $v_2$ or $v_3$. In this case, we get $T$ isomorphic to $T_{17}$, which is shown in Figure~\ref{f1}. It can be seen that $T$ is $C_4$-free and $e(H)=26$.   
\end{proof}
\begin{claim}\label{cl7}
There is no $H$ when $k=4$.
\end{claim}
\begin{proof}
In this case $S=\{v_0,v_0',v_1,v_2,u,u_1,u_2,v_4,v_4',v_3\}$. $|S|=10.$ There are 7 vertices, say $x_1,x_2\dots, x_7$ not not labeled in $T$. None of these vertices is adjacent to $v_1$ or $v_3$. In other words, if any of the seven vertices is adjacent to a vertex in $L$, then it is with $v_2$. It can be seen that there is a vertex in $\{x_1,x_2,\dots,x_7\}$, which is adjacent to $v_2$ and is a pendant in $T$. Suppose $x_1$ is such a vertex. By the choice of the path $L$, $x_1v_2$ can not be incident to a face of size at least 5. In other words, the two faces incident to the edge are either a $3$-face or a $4$-face. But in any possibility, $H$ contains a $C_4$, which is a contradiction. This completes the proof of Claim~\ref{cl7} and Lemma~\ref{vbv}.    
\end{proof}
\end{proof}
\begin{lemma}\label{nbn}
For an $18$-vertex $C_4$-free Halin graph $H$, $e(H)\leq 28$, i.e., $e(H)\leq \frac{5}{3}(n-3)+3$ where $n=18.$
\end{lemma}
\begin{proof}
Let $T$ be the characteristic tree of $H$, and $L=(v_0,\ v_1,\ v_2,\ \dots,\ v_k)$ be a longest path in $T$. It can be checked that $k\geq 4.$ 
\begin{claim}\label{cl8}
$k$ is at most 7. 
\end{claim}
\begin{proof}
By Observation~\ref{ob1}, $|S|=14$. We remain four vertices that are not labeled yet. If three of the vertices are already in $L$, then at least two vertices, which are non-leaf, become degree-$2$. This is a contradiction and therefore $k\leq 7$. This completes the proof of Claim~\ref{cl8}.   
\end{proof}
\begin{claim}\label{cl101}
If $k=7$, $e(H)=27$.    
\end{claim}
\begin{proof}
Here $S=\{v_0,v_0',v_1,v_2,u,u_1,u_2,v_7,v_7',v_6,v_5,w,w_1,w_2\}$. Observe that $v_3$ and $v_4$ are vertices in $L$ but not addressed yet, as the vertices are degree-$2$ and should be with the degree at least 3. There are two remaining vertices, say $x_1$ and $x_2$, which are not on $L$ but in $T$. From the degree condition of a Halin graph, $v_3$ and $v_4$ are adjacent to only one vertex in $\{x_1,x_2\}$. Let $x_1\in N_T(v_3)$, and  $x_2\in N_T(v_4)$. Notice that the two edges can not be incident to the same face. Otherwise, the $x_1x_2\in E(H)$ and we get $4$-cycle $(v_4,\ v_3,\ x_1,\ x_2,\ v_4)$, which is a contradiction. Therefore the two edges must be on opposite sides of $L$ in the planar embedding of $H$. $H$ which is shown in Figure~\ref{f5} is the only Halin graph with such property. It can be seen that $H$ is $C_4$-free and $e(H)=e(T)+10=27.$ This completes the proof of Claim~\ref{cl101}. 
\begin{figure}[h]
\centering
\begin{tikzpicture}[scale=0.12]
\draw[ultra thick](5,0)--(15,0)--(23.1,5.9)--(28.4,6.8)(23.1,5.9)--(25.6,10.6) (15,0)--(23.1,-5.9)--(28.4,-6.8)(23.1,-5.9)--(25.6,-10.6) (-23.1,5.9)--(-28.4,6.8)(-23.1,5.9)--(-25.6,10.6) (-23.1,-5.9)--(-25.6,-10.6)(-23.1,-5.9)--(-28.4,-6.8)(-5,0)--(-15,0)--(-23.1,5.9)(-15,0)--(-23.1,-5.9)(5,10)--(5,0)(-5,-10)--(-5,0)--(5,0);
\draw[fill=white](5,0)circle(25pt);
\draw[fill=white](5,10)circle(25pt);
\draw[fill=white](-5,-10)circle(25pt);
\draw[fill=white](-15,0)circle(30pt);
\draw[fill=white](-5,0)circle(25pt);
\draw[fill=white](15,0)circle(30pt);
\draw[fill=white](23.1,5.9)circle(25pt);
\draw[fill=white](23.1,-5.9)circle(25pt);
\draw[fill=white](28.4,6.8)circle(25pt);
\draw[fill=white](25.6,10.6)circle(25pt);
\draw[fill=white](28.4,-6.8)circle(25pt);
\draw[fill=white](25.6,-10.6)circle(25pt);
\draw[fill=white](-23.1,5.9)circle(25pt);
\draw[fill=white](-23.1,-5.9)circle(25pt);
\draw[fill=white](-28.4,6.8)circle(25pt);
\draw[fill=white](-25.6,10.6)circle(25pt);
\draw[fill=white](-28.4,-6.8)circle(25pt);
\draw[fill=white](-25.6,-10.6)circle(25pt);
\end{tikzpicture}
\caption{Characteristic tree of an $18$-vertex Halin graph}
\label{f5}
\end{figure}
\end{proof}

\begin{claim}\label{cl10}
If $k=6$, $e(H)\leq 28$.    
\end{claim}
\begin{proof}
$S=\{v_0,v_0',v_1,v_2,u,u_1,u_2,v_6,v_6',v_5,v_4,w,w_1,w_2\}$ and $v_3$ is a non-leaf in $L$. Moreover, there are three vertices, say $x_1,x_2$, and $x_3$ which are not in $L$. Thus, $v_3$ must be adjacent to one of the three vertices. If any of the remaining three vertices is adjacent to a vertex in $\{v_2,v_4\}$, it is easy to get a $4$-cycle, which is a contradiction. We have two possible graphs.

The first graph is when all the three vertices, $x_1,x_2$, and $x_3$ are adjacent to $v_3$. It is easy to see that, not all edges, $x_1v_3, x_2v_3$, and $x_3v_3$, are on the same side of $L$. Otherwise, $H$ contains a $C_4$. In this case, the characteristic tree is $T_{18}$ and is shown in Figure~\ref{f1}. Clearly $e(h)=e(T)+11=28.$

The second graph is obtained by identifying the dark-spotted vertex of the star with the $v_3$. The graph is shown in Figure~\ref{f16}. Here $e(H)=e(T)+10=27.$ This completes the proof of Claim~\ref{cl10}
\begin{figure}[h]
\centering
\begin{tikzpicture}[scale=0.12]
\draw[ultra thick](-5,0)--(5,0)--(13.1,5.9)--(18.4,6.8)(13.1,5.9)--(15.6,10.6) (5,0)--(13.1,-5.9)--(18.4,-6.8)(13.1,-5.9)--(15.6,-10.6) (-23.1,5.9)--(-28.4,6.8)(-23.1,5.9)--(-25.6,10.6) (-23.1,-5.9)--(-25.6,-10.6)(-23.1,-5.9)--(-28.4,-6.8)(-5,0)--(-15,0)--(-23.1,5.9)(-15,0)--(-23.1,-5.9)(-5,10)--(0.9,18.1)(-5,10)--(-10.9,18.1)(-5,0)--(-5,10);
\draw[fill=white](-5,10)circle(25pt);
\draw[fill=white](-15,0)circle(30pt);
\draw[fill=white](-5,0)circle(25pt);
\draw[fill=white](5,0)circle(30pt);
\draw[fill=white](13.1,5.9)circle(25pt);
\draw[fill=white](13.1,-5.9)circle(25pt);
\draw[fill=white](18.4,6.8)circle(25pt);
\draw[fill=white](15.6,10.6)circle(25pt);
\draw[fill=white](18.4,-6.8)circle(25pt);
\draw[fill=white](15.6,-10.6)circle(25pt);
\draw[fill=white](-23.1,5.9)circle(25pt);
\draw[fill=white](-23.1,-5.9)circle(25pt);
\draw[fill=white](-28.4,6.8)circle(25pt);
\draw[fill=white](-25.6,10.6)circle(25pt);
\draw[fill=white](-28.4,-6.8)circle(25pt);
\draw[fill=white](-25.6,-10.6)circle(25pt);
\draw[fill=white](0.9,18.1)circle(25pt);
\draw[fill=white](-10.9,18.1)circle(25pt);
\end{tikzpicture}
\caption{Characteristic tree of an $18$-vertex Halin graph}
\label{f16}
\end{figure}    
\end{proof}
\begin{claim}\label{cl13}
There is no $H$ when $k=4$ or $5$.
\end{claim}
\begin{proof}
Let $k=5$. In this case we have, $S=\{v_0,v_0',v_1,v_2,u,u_1,u_2,v_5,v_5',v_4,v_3,w,w_1,w_2\}$. and there are four vertices, say $x_1,\ x_2,\ x_3$ and $x_4$, which are not in $L$. If any of the vertices is incident to $L$, then it must be adjacent to either $v_2$ or $v_3$. Without loss of generality let $x_1$ be adjacent to $v_2$. Then the two faces which are incident to the edge $x_1v_2$ are either both $3$-face or both $4$-face or a mix of the two. But in all three cases, $H$ contains a $4$-cycle, which is a contradiction. Therefore, $x_1$ must be adjacent to two pendant vertices, say $x_2$ and $x_3$. However, this results in a vertex $x_4$ which is not incident to any of the vertex in $L$. Which is a contradiction as $T$ is a tree.  

A similar argument can be given to show that we do not have an $18$-vertex $C_4$-free Halin graph such that the characteristic tree has the longest path of length 4. This completes the proof of Claim~\ref{cl13} and Lemma~\ref{nbn}.   
\end{proof}
\end{proof}
\section{Conjectures and concluding remarks}
As mentioned earlier in the beginning, Bondy and Lov\'asz proved that a Halin graph is pancyclic if every non-leaf in its characteristic tree is of degree at least 4. It is also remarked that, if the characteristic tree contains a vertex of degree three, cycles of all lengths will still be in the graph with a possible exception of an even-length cycle. The following is our conjecture concerning the sharp upper bound of the Halin Tur\'an number of the $6$-cycle. 
\begin{conjecture}
For $n\geq 21,$ $$\ex_{\hh}(n,C_6)\leq \frac{8}{5}(n-1).$$
\end{conjecture}
%An extremal construction realizing the bound can be obtained by identifying the black-spotted vertex of $\frac{n-1}{10}$ copies of the tree shown in Figure~\ref{ghb}~(left). The corresponding Halin graph is also shown in Figure~\ref{ghb}~(right) when we use only three copies of the tree. It can be checked that the graph contains no $6$-cycle and $e(H)=\frac{8}{5}(n-1)$.  
%\begin{figure}[ht]
%\centering
%\includegraphics[scale=0.85]{pic.png}
%\caption{A $C_6$-free Halin graph and its building structure.}
%\label{ghb}
%\end{figure}
\section*{Dedication and last word}
The author wishes to dedicate this research in memory of the young \textbf{AMHARA FANOS} who bravely gave their lives for their people, who have endured government-led ethnic cleansing and genocidal attacks over the past five years. In line with this, the author calls upon nations and the global community to stay informed about the current state of the country and not be influenced by the government's deceptive \say{prosperity} narratives.

\section*{Conflict of interest}
The author declares no conflict of interest.

\end{document}